\documentclass[a4paper,12pt]{article}
\usepackage[margin=1in]{geometry}
\usepackage{setspace}
\usepackage[english]{babel}
\usepackage[utf8]{inputenc}
\usepackage{amsmath,amssymb,amsfonts}
\usepackage{amsthm}
\usepackage{hyperref}
\newtheorem{theorem}{Theorem}[section]
\newtheorem{definition}{Definition}[section]
\newtheorem{example}{Example}[section]

\newtheorem{proposition}{Proposition}[section]

\title{A Complementarity Property of Rank-One Multiplicative Maps}
\date{}
\author{K.C. Sivakumar \\
Department of Mathematics\\
Indian Institute of Technology Madras\\
Chennai, 600036, India} 

\begin{document}
\maketitle

\begin{abstract}
The semidefinite linear complementarity problem associated with the multiplicative map $M_A:\mathbb{S}^n\to\mathbb{S}^n$ defined by $M_A(X)=AXA^T$, is considered, in the case when $A\in\mathbb{R}^{n\times n}$ has rank one. A linear map is said to have the $Q_0$-property if every feasible instance of the associated complementarity problem admits a complementary solution.  We prove two sharp, complementary results. When $A=xy^T$ with $x,y\in\mathbb{R}^n$ \emph{linearly independent}, we construct a matrix $Q\in\mathbb{S}^n$ for which the associated semidefinite complementarity problem is feasible, but possesses no complementary solution; proving that $M_A$ \emph{does not have} the $Q_0$-property.  Conversely, when $A=uu^T$ for some nonzero $u\in\mathbb{R}^n$, we prove that $M_A$ \emph{has} the $Q_0$-property. These two results provide a complete characterization of the $Q_0$-property within the family of rank-one multiplicative maps.
\end{abstract}



\newpage

\section{Introduction}
We study the semidefinite linear complementarity problem associated with the multiplicative map
\(M_A : \mathbb{S}^n \to \mathbb{S}^n\) defined by \(M_A(X) = AXA^T\), in the case where \(A \in \mathbb{R}^{n \times n}\) is a rank-one matrix. A linear map is said to have the \(Q_0\)-property, if every feasible instance of its associated semidefinite complementarity problem admits a complementary solution. These notions will be made precise in the next section. In this setting, we establish two sharp and complementary results. First, when \(A = xy^T\) with linearly independent vectors \(x,y \in \mathbb{R}^n\), we construct a matrix \(Q \in \mathbb{S}^n\) such that the corresponding semidefinite
complementarity problem is feasible,  but has no complementary solution, showing that \(M_A\) fails to have the \(Q_0\)-property. On the other hand, when \(A = uu^T\) for some nonzero \(u \in \mathbb{R}^n\), we prove that \(M_A\) does possess the \(Q_0\)-property. Taken together, these results yield a complete characterization of the \(Q_0\)-property within the class of rank-one multiplicative maps. 

\section{Framework, Notation  and Preliminaries}
\subsection{The linear complementarity problem}
Let $\mathbb{R}^{n\times n}$ denote the set of all square matrices of order $n,$ with real entries, while $\mathbb{R}^n$ is the $n$-dimensional real Euclidean space. For $M\in\mathbb{R}^{n\times n}$ and $q\in\mathbb{R}^n$, the {\it linear complementarity problem} (LCP$(M,q)$) is to determine if there is a vector
$x\in\mathbb{R}^n$ satisfying
\begin{center}
$x\geq 0,\quad Mx+q\geq 0~~$ and $~~x^T(Mx+q)=0,$    
\end{center}
or to show that such a vector does not exist. Here, $x \geq 0$ signifies that all the coordinates of $x$ are nonnegative. A vector $x$ satisfying the first two linear inequalities above is called a {\it feasible} vector. If a feasible vector also satisfies the complementarity condition, then it is called a {\it solution}. The theory of the LCP, including a rich taxonomy of matrix classes characterized by their complementarity behaviour, is thoroughly developed in the monograph \cite{CPS1992}. We also refer the reader to the more recent article \cite{cott} presenting an excellent summary.

Among the most studied matrix classes in LCP theory are the $Q$-matrices and the $Q_0$-matrices.  A matrix $M$ is called a \emph{$Q$-matrix} (or $M$ has the $Q$-property) if LCP$(M,q)$ has a solution for every $q\in\mathbb{R}^n$, and a \emph{$Q_0$-matrix} (or $M$ has the $Q_0$-property) if every \emph{feasible} instance of LCP$(M,q)$ has a {\it solution}, without demanding that all instances be solvable. We informally refer to this as ``feasibility implies solvability." Clearly, every $Q$-matrix is a $Q_0$-matrix, but the converse fails. In view of this, when we address the problem of determining if a matrix is a $Q_0$-matrix, then we tacitly assume that it is not a $Q$-matrix. It is pertinent to observe that LCP$(M,q)$ always has a solution (viz., $x=0$) when $q \geq 0.$ Thus, when we want to verify if a matrix has the $Q_0$-property, then we may assume that the vector $q$ is not nonnegative. 

Let us recall the following results for (entrywise) nonnegative matrices, in the context of the LCP. This provides the main  motivation for our study. 

\begin{enumerate}
    \item \textbf{$Q$-Property for nonnegative matrices:} A nonnegative matrix is a $Q$-matrix if and only if all its diagonal entries are positive \cite[Theorem 5.2]{mur}.
    \item \textbf{$Q$-Property for rank-one matrices:} A rank-one matrix is a $Q$-matrix if and only if all its entries are positive \cite[Theorem 3.4]{kcs}.
    \item \textbf{$Q_0$-Property for nonnegative matrices:} A nonnegative matrix is a $Q_0$-matrix if and only if any row containing a zero diagonal entry has all its entries zero \cite[Theorem 2.5]{survey}. 
\end{enumerate}

Interestingly, the question of characterizing \textbf{$Q_0$}-property for rank-one matrices does not seem to have been investigated.

\subsection{The semidefinite linear complementarity problem}

We write $\mathbb{S}^n$ for the space of real symmetric $n\times n$ matrices, equipped with the trace inner product $\langle X,Y\rangle:=\operatorname{tr}(XY)$.  We write $X\succeq 0$ (resp.\ $X\succ 0$) to mean that $X$ is symmetric and positive semidefinite (resp.\ positive definite), and $\mathbb{S}^n_+$ for the cone of positive semidefinite matrices.  For a matrix $W$, $\mathcal{R}(W)$ and $\mathcal{N}(W)$ denote its range space and null space, respectively. The notation $u^\perp$ stands for the orthogonal complement of the
singleton $\{u\}$ in $\mathbb{R}^n$ and ``$\rm sp$" denotes the span. Finally, $A^T$ denotes the transpose of the matrix $A$.

Next, we define the semidefinite linear complementarity problem (SDLCP). This problem may be thought of a matrix-valued generalization of the LCP, in which the nonnegative orthant of
$\mathbb{R}^n$ is replaced by the cone $\mathbb{S}^n_+$.

\begin{definition}
Let $\mathcal{L}:\mathbb{S}^n\to\mathbb{S}^n$ be a linear map and $Q\in\mathbb{S}^n$.  The \emph{semidefinite linear complementarity problem} $\mathrm{SDLCP}(\mathcal{L},Q)$ is to find matrices $X,W\in\mathbb{S}^n$
such that
\[
  X\succeq 0,\qquad W = \mathcal{L}(X)+Q\succeq 0,\quad \textit{and} \quad \langle X,W\rangle = 0,
\]
\end{definition}
or to show that such matrices do not exist. 
Again, the first two conditions are linear inequality requirements, while the third condition is a {\it complementarity} condition. We say that $\mathrm{SDLCP}(\mathcal{L},Q)$ is \emph{feasible} if there
exists $X\succeq 0$ with $\mathcal{L}(X)+Q\succeq 0$; such a matrix $X$ also being called {\it feasible}. The problem $\mathrm{SDLCP}(\mathcal{L},Q)$ is \emph{solvable} if there
exists a feasible $X$ which also satisfies the complementarity condition. 

\begin{definition}
The linear map
$\mathcal{L}$ is said to have the \emph{$Q$-property} if for every $Q\in\mathbb{S}^n$, the problem $\mathrm{SDLCP}(\mathcal{L},Q)$ has a solution. The linear map $\mathcal{L}$ is said to have the \emph{$Q_0$-property} if for every
$Q\in\mathbb{S}^n$, feasibility of $\mathrm{SDLCP}(\mathcal{L},Q)$
implies that it is solvable.    
\end{definition}

We make the following observations:

\begin{enumerate}
    \item The complementarity condition $\langle X,W\rangle=0$ combined with $X\succeq 0$ and $W\succeq 0$ is equivalent to $XW=0$ (equivalently $WX=0$), since for positive semidefinite matrices $\operatorname{tr}(XW)=0\iff XW=0$. This is further equivalent to ${\cal R}(W) \subseteq {\cal N}(X).$
\item As is the case for the LCP, when we refer to a linear map ${\cal L}$ as having the $Q_0$-property, we assume that it does not have the $Q$-property. 
\item Again, as is the case for the LCP, when we are verifying if ${\cal L}$ has the $Q_0$-property, we may assume that $Q \nsucceq 0,$ whenever $\mathrm{SDLCP}(\mathcal{L},Q)$ is feasible.
\end{enumerate}

Here is a brief survey. Let us recall that a  linear map $\mathcal{L}$ is said to have the \emph{$P$-property} if $$X{\mathcal L}(X)={\mathcal L}(X)X \preceq 0 \Longrightarrow X=0.$$
This definition was proposed by Gowda and Song \cite{GS2000} motivated the sign non-reversal property of $P$-matrices. They carried out a detailed study of the SDLCP analogous to the classical LCP theory. The concepts of $P$-matrix, $Q$-matrix and the globally uniquely solvability (GUS) property for the ordinary linear complementarity problem were generalized and existence results were obtained. The classical Lyapunov theorem was presented as a special case of an equivalence theorem that shows that the $P$-property and $Q$-property  of the Lyapunov transformation $L_A$ defined by $L_A(X):=AX+XA^T,$ are equivalent to the condition that all real parts of the eigenvalues of $A$ are positive. Gowda and Parthasarathy \cite{GP2000} subsequently established complementarity forms of the theorems of Lyapunov and Stein, further illuminating the structure of the SDLCP. In particular, they showed that the $P$-property of the Stein transformation $S_A,$ defined on the space of complex hermitian matrices, by $S_A(X):=X-AXA^*$, is equivalent to the condition that all the eigenvalues of $A$ lie in the open unit disc. 

Another linear map, called the multiplicative map, has been investigated quite extensively in the SDLCP literature, concerning its $Q$-property. In this article, we shall be interested in the $Q_0$-property of the multiplicative map, which we recall, next. 

\subsection{The multiplicative map}
The \emph{multiplicative map} $M_A:\mathbb{S}^n\to\mathbb{S}^n,$ associated with $A\in\mathbb{R}^{n\times n}$ is defined by
\[
  M_A(X)=AXA^T.
\]
When $A$ has rank one, say $A=xy^T$, then $M_A(X)$ is a rank one matrix, for every $X$:
\[
  M_A(X) = (y^TXy)\,xx^T.
\]
Conversely, if $M_A$ is a rank-one operator, then $M_A(I)=AA^T$ is of rank one and so $A$ is of rank one.

\begin{definition}
A linear transformation $L:\mathbb{S}^n \to \mathbb{S}^n$ is a \emph{positive operator} with respect to the positive semidefinite cone $\mathbb{S}^n_+$ if $L(\mathbb{S}^n_+) \subseteq \mathbb{S}^n_+$, meaning:
\begin{equation}
X \succeq 0 \implies L(X) \succeq 0.
\end{equation}
\end{definition}

For any vector $v \in \mathbb{R}^n$, and $X \in \mathbb{S}^n_+$:
$$v^T M_A(X) v = v^T (AXA^T) v = (A^T v)^T X (A^T v) \geq 0.$$ Thus, 
$M_A$ is a positive operator. Next, let $M_A$ possess the $Q$-property. Let $A^Tx=0$ and set $Q:=-xx^T$. By the $Q$-property of $M_A$, there exists $X \succeq 0$ such that $AXA^T-xx^T=M_A(X) +Q\succeq 0,$ i.e. $$0 \leq x^TAXA^Tx-{\parallel x \parallel}^4=-{\parallel x \parallel}^4,$$ so that $x=0.$ Here, $\parallel . \parallel$ refers to the Euclidean norm on $\mathbb{R}^n.$ We summarize these, in the next result.

\begin{proposition}\label{prop:MA_pos}
Let $A \in \mathbb{R}^{n \times n}$. We have: 
\begin{enumerate}
    \item $M_A$ is a positive operator on $\mathbb{S}^n_+$.
\item If $M_A$ has the $Q$-property, then $A$ is nonsingular.
\end{enumerate}
\end{proposition}

Let $A$ be a rank-one matrix. From the second item of Proposition \ref{prop:MA_pos} in particular, it follows that $M_A$ does not have the $Q$-property. Thus, it is pertinent to ask if it has the $Q_0$-property. In view of the first item, this problem may be thought of as a semidefinite version of the results, known for nonnegative matrices, in the standard LCP. As mentioned earlier, while we are not aware of any result that characterizes when a nonnegative matrix has the $Q_0$-property (concerning the standard LCP), we give a complete answer to this question in the SDLCP setting, for the operator $M_A$.

Here is a quick summary of the results of this article. 

\begin{itemize}
  \item \textbf{Theorem~\ref{thm:notQ0}}  shows
    that if $A=xy^T$ with $x,y$ linearly \emph{independent}, then $M_A$
    does \emph{not} have the $Q_0$-property. The adopted proof method gives rise to two cases, and for each possibility, the proof constructs an explicit matrix $Q$ for which the SDLCP is feasible, but admits no     complementary solution.
  \item \textbf{Theorem~\ref{symmetric}} shows that
    if $A=uu^T$, so that $A$ is symmetric positive semidefinite of     rank one, then $M_A$ \emph{has} the $Q_0$-property.  
    Here, for any matrix $Q$ for which the SDLCP is feasible, 
    we explicitly construct a complementary solution. Interestingly, the solution is explicit and involves the Moore-Penrose inverse of $Q$.
\end{itemize}

Thus, within the class of rank-one multiplicative maps, the $Q_0$-property
holds precisely when $A$ is a symmetric positive semidefinite rank-one matrix, i.e.\ when $A=uu^T$ for some nonzero $u$.

Let us give a short survey of the literature on the map $M_A.$ It was shown in \cite{bhim} that $M_A$ has the GUS property if and only if $A$ is either positive definite or negative definite. It was proved in \cite{tp} that, if $A$ is symmetric, then $M_A$ has the $Q$-property if and only if $A$ is either positive definite or negative definite. Generalizations of some of these results for Euclidean Jordan Algebras were obtained in \cite{gowszn}.

We conclude this section recalling the notion of the Moore-Penrose inverse. For $A \in \mathbb{R}^{n \times n}$, there is a unique $X\in \mathbb{R}^{n \times n}$ satisfying the matrix equations: $AXA=A,~XAX=X,~(AX)^T=AX$ and $(XA)^T=XA.$ This unique $X$ is denoted by $A^{\dag}$ and is called the Moore-Penrose inverse of $A$. Note that, if $D=\rm diag(d_1, d_2, \ldots, d_n),$ the diagonal matrix with diagonal entries $d_1, d_2, \ldots, d_n,$ then $D^{\dag}=\rm diag(d_1^{\dag}, d_2^{\dag}, \ldots, d_n^{\dag}),$ where for any number $\lambda,$ we have ${\lambda}^{\dag}=\frac{1}{\lambda},$ if $\lambda \neq 0$, while  ${\lambda}^{\dag}=0,$ if $\lambda=0.$ Finally, if $A$ is symmetric with the eigen-decomposition $A=UDU^T$, for an orthogonal matrix $U$ and $D$ a diagonal matrix, then $A^{\dag}=UD^{\dag}U^T.$ For more details, we refer to \cite{bengre}.

Finally, if $A$ is a real symmetric matrix, then ${\cal R}(A)$ and ${\cal N}(A)$ are orthogonal complementary subspaces of $\mathbb{R}^n.$ We shall be making an implicit use of this property.

\section{$M_A$ is not a $Q_0$-operator, when $A$ is rank one nonsymmetric.}
In this section, we show that when $A$ is a non-symmetric rank one matrix, $M_A$ does not possess the $Q_0$-property. 

\begin{theorem}\label{thm:notQ0}
For linearly independent vectors $x,y\in\mathbb{R}^n$, let $A = xy^T$. Then $M_A$ does not have the $Q_0$-property.
\end{theorem}
\begin{proof}
Since $x,y$ are linearly independent, they are nonzero. Set 
\[
u := x/\|x\| \quad \textit{and} \quad v := \frac{(I-uu^T)y}{\|(I-uu^T)y\|}.
\]
Since $y\notin \rm sp\{x\}$ we have $(I-uu^T)y\neq 0$, so that $v$ is well defined. Also, $u^Tv=0$. Set $\beta_1:=u^Ty$ and $\beta_2:= \|(I-uu^T)y\| > 0.$ Then, 
\[
  y = \beta_1 u + \beta_2 v.
\tag{$\ast$}
\] 
For any $X\in\mathbb{S}^n$, $M_A(X) = x(y^TXy)x^T$ and so
$\mathcal{R}(M_A) = \rm sp \{uu^T\}$.
While the matrix $M_A(X)$ remains the same if one replaces $y$ by $-y$, the scalar $\beta_1$ changes to $-\beta_1$. Hence, we may assume without loss of generality, that $\beta_1\leq 0$. We consider the two cases on $\beta_1.$

\medskip
\noindent\textbf{Case~1: $\beta_1=0$.}\\
Then, $y=\beta_2 v$. Fix numbers $\alpha>0, ~\eta>0$ and set
$$Q:=-\alpha\,uu^T+\eta\,vv^T.$$
Since $u^TQu=-\alpha<0$ and $v^TQv=\eta>0$, the matrix $Q$ is indefinite. For $t>0,$ set 
\[
Z_t:=t(u-v)(u-v)^T\succeq 0,
\tag{$\ast\ast$} \]
so that 
$$y^TZ_ty = t{\beta_2}^2((u-v)^Tv)^2= t\beta_2^2.$$
Then, $$M_A(Z_t)=(y^TZ_ty)xx^T=t{\beta_2}^2 xx^T=\lambda t uu^T,$$
where $\lambda:=\beta_2^2\|x\|^2>0.$ Thus, if $t$ is chosen such that $t\geq \frac{\alpha}{\lambda} >0,$ then 
\[
  W_t:=M_A(Z_t)+Q = (\lambda t-\alpha)\,uu^T+\eta\,vv^T\succeq 0.
\]
We have shown that, $Z_t$ is feasible for $\mathrm{SDLCP}(M_A,Q)$. Next, for the chosen $Q$, we show that there is no complementary solution. Observe that 
\begin{eqnarray*}
M_A(X)+Q & = & (y^TXy)xx^T-\alpha uu^T+ \eta vv^T \\ 
& = & ({\|x \|}^2 y^TXy - \alpha)uu^T + \eta vv^T.    
\end{eqnarray*}
Suppose that there is a complementary solution $X$, so that, $X\succeq 0$ and 
$$M_A(X)+Q =(\|x\|^2 y^TXy-\alpha)\,uu^T+\eta\,vv^T=:W \succeq 0,$$ with $\langle X,W\rangle=0.$ Note that, 
$0 \leq u^TWu=\|x\|^2 y^TXy-\alpha.$ Also, 
\[
\langle uu^T,X \rangle = \operatorname{tr}(uu^TX) = \operatorname{tr}(u^TXu) = u^TXu.
\]
Thus, 
\[
0=  \langle X,W\rangle
  = \underbrace{(\|x\|^2 y^TXy-\alpha)}_{\geq\,0}\underbrace{u^TXu}_{\geq\,0}
  +\,\eta\underbrace{v^TXv}_{\geq\,0}.
\]
In particular, $v^TXv=0$ and since $X\succeq 0$, this means that $Xv=0$. Thus, $y^TXy=0$ since $y=\beta_2 v$, implying that the first factor in the first term above (namely $-\alpha$), is negative, a contradiction. 

\medskip
\noindent\textbf{Case~2: $\beta_1<0$.}\\
Observe that, $\beta_1\neq\beta_2.$ Choose $\alpha >0$ and set $\eta:=-\beta_1/\beta_2>0.$ Define $$Q:=-\alpha\,uu^T-\eta(uv^T+vu^T)+vv^T.$$
Again, $Q$ is indefinite, since $u^TQu=-\alpha<0$ and $v^TQv=1>0$. Let $Z_t \succeq 0$ be as given in ($\ast \ast$).  Considering $y$ as in ($\ast$), we have 
\begin{eqnarray*}
  y^TZ_ty & = & 
  t((u-v)^Ty)^2 \\
  & = & t((u^Ty-v^Ty)^2 \\
  & = & t(\beta_1-\beta_2)^2 \\
  & = & t\gamma > 0,  
\end{eqnarray*}
with $\gamma:=(\beta_1-\beta_2)^2.$
If we set $\lambda:=\gamma\|x\|^2>0$, then 
\[
M_A(Z_t) = t \gamma xx^T= t\gamma\|x\|^2 uu^T=\lambda t uu^T,
\]
so that, 
\[
  W_t := M_A(Z_t)+Q = (\lambda t-\alpha)\,uu^T-\eta(uv^T+vu^T)+vv^T.
\]
We claim that $W_t$ is positive semidefinite. Note that, for every $z \in \rm sp\{u,v\}^{\perp}$, $W_tz=0$. Thus, $W_t$ is positive semidefinite if and only if it is positive semidefinite on $\rm sp\{u,v\}$. In order to show the latter, let ${\bf x}_s:=u+sv$. Then, \[{\bf x}_s^TW_t{\bf x}_s = s^2-2\eta s+(\lambda t-\alpha).\]
This is nonnegative for all $s\in\mathbb{R},$ if and only if $t\geq(\alpha+\eta^2)/\lambda$.
Thus $W_t\succeq 0$ for any choice of such a $t$, establishing feasibility.

We show, by contradiction, that for the constructed matrix $Q$, there is no complementary solution. Suppose that there exists $X\succeq 0,~W:=M_A(X)+Q\succeq 0$ with $\langle X,W\rangle=0.$ Then, $XW=0=WX$. 
Setting
$M_A(X)=\delta \,uu^T$ with $\delta:=\|x\|^2(y^TXy)\geq 0$, we have 
$$
    W=(\delta-\alpha)\,uu^T-\eta(uv^T+vu^T)+vv^T.
$$
Observe that ${\cal R}(W)$ is of dimension at most $2$. The restriction of $W$ to the subspace $\operatorname{span}\{u,v\},$ relative to the basis $\{u,v\},$ has the matrix representation $$M_\delta:=\begin{pmatrix}\delta-\alpha&-\eta\\-\eta&1\end{pmatrix}.$$ Since $W\succeq 0,$  we have $0 \leq \det(M_\delta)=\delta-\alpha-\eta^2$, so that $\delta\geq\alpha+\eta^2$. We show that $M_A(X)=0,$ which would imply that  $W=Q\not\succeq 0$, a contradiction. We consider the two cases on $\delta.$

Let $\delta>\alpha+\eta^2$.
Then, $M_\delta$ is invertible and so $y \in \operatorname{span}\{u,v\}={\mathcal{R}(W)}$. Note that, since $XW=0,$ we have $\mathcal{R}(W)\subseteq N(X).$ Thus, $Xy=0$, proving that $M_A(X)=0$. 

Finally, suppose that $\delta=\alpha+\eta^2$. 
Here, $$W = \eta^2 uu^T-\eta(uv^T+vu^T)+vv^T = (\eta u-v)(\eta u-v)^T,$$
so that $\mathcal{R}(W)=\operatorname{span}\{\eta u-v\}$.
Again, since $\mathcal{R}(W)\subseteq N(X),$ we have $X(\eta u-v)=0$, i.e. $Xv=\eta Xu$. This leads to the expressions 
\begin{center}
$u^TXv=\eta(u^TXu)$ and $v^TXv=\eta (v^TXu)=\eta^2(u^TXu)$    
\end{center}
and so 
\[
  y^TXy = (\beta_1 u+\beta_2 v)^TX(\beta_1 u+\beta_2 v)
  = (\beta_1+\eta\beta_2)^2(u^TXu).
\]
Since $\eta=-\beta_1/\beta_2$, we have 
$y^TXy=0$ and so $M_A(X) = x(y^TXy)x^T=0$, again. \\

This completes the proof that $M_A$ does not have the $Q_0$-property.
\end{proof}

\section{$M_A$ has $Q_0$-property, when $A$ is symmetric rank-one.}

Here, we prove that when $A$ is symmetric and is of rank one, $M_A$ has the $Q_0$-property. In other words, we show that the converse of Theorem \ref{thm:notQ0} is true. We make some preliminary considerations towards proving it. The key idea involves a minimization problem.

Assuming that the symmetric matrix $Q$ is not positive semidefinite and $u \neq 0,$ we define the set
\[
  \Gamma_{Q,u} := \bigl\{\lambda \in \mathbb{R}: \lambda\,uu^T+Q\succeq 0\bigr\}.
\]
If $\Gamma_{Q,u} \neq \emptyset,$ then it is a closed ray. It is upward closed: if $\lambda\in\Gamma_{Q,u}$ and $\lambda'>\lambda,$ then $\lambda'uu^T+Q=(\lambda uu^T+Q)+(\lambda'-\lambda)uu^T\succeq 0.$ By item 1 of Theorem \ref{qprop1} below, $\Gamma_{Q,u}$ is bounded below by $0.$ Hence
\[
\Gamma_{Q,u}=[\lambda_0,\infty), \qquad \text{where } \lambda_0:=\min\Gamma_{Q,u}>0.
\]
Now, consider the \textit{semidefinite programming (SDP)} problem:
\[
\begin{aligned}
\text{Minimize} \qquad & \lambda \\
\text{subject to} \qquad & \lambda \in \Gamma_{Q,u}.
\end{aligned}
\tag{1}
\]
We show that this SDP admits an explicit (optimal) solution, in Theorem \ref{mugamma}. This solution is expressed in terms of the Moore-Penrose inverse of $Q$ (and the vector $u$). We need some intermediate properties, that are collected in the next result.

\begin{theorem}\label{qprop1}
Assume that $Q \nsucceq 0,$ $u \neq 0$ and
$\Gamma_{Q,u} \neq \emptyset.$ Then, the following hold:
\begin{enumerate}
\item $\lambda \in \Gamma_{Q,u} \Longrightarrow \lambda >0.$
\item $u \in {\cal R}(Q)$.
\item $Q$ has exactly one negative eigenvalue, and this eigenvalue is simple.
\end{enumerate}
\end{theorem}
\begin{proof}
1. Since $Q \nsucceq 0$, it has a negative eigenvalue. Let $v$ be such that $Qv = -\mu v,~\|v \|=1$ and $\mu > 0$. Let $\lambda\in\Gamma_{Q,u}$ so that $W=Q+\lambda uu^T\succeq0$. Then,
$$0 \leq v^TWv=v^TQv+\lambda {(u^Tv)}^2=-\mu +\lambda {(u^Tv)}^2.$$
Note that, if $u^Tv=0,$ then $\mu \leq 0,$ a contradiction. Thus, ${(u^Tv)}^2 >0$ so that $\lambda\geq\mu/{(u^Tv)}^2>0$.

2. Let $v$ be as above, so that $v \in {\cal R}(Q)={\cal N}(Q)^{\perp}$. Suppose that $u \notin {\cal R}(Q)$. Consider the decomposition: 
\begin{center}
$u = y + z~$ with $~y \in {\cal R}(Q)~$ and $~0 \neq z \in {\cal N}(Q).$    
\end{center}
Then $Qz=0$ and $y^Tz=0,$ so that $u^Tz = \|z\|^2 > 0.$ Also, $v^Tz=0.$ Next, let $\lambda \in \Gamma_{Q,u}$ so that $W: = Q+\lambda uu^T \succeq 0$. We have:
\[
v^TWv = -\mu+\lambda(u^Tv)^2,\qquad
z^TWz = \lambda\|z\|^4,\qquad
v^TWz = \lambda(u^Tv)\|z\|^2.
\]
Since $W\succeq 0$, the Gram matrix of the form $(a,b)\mapsto a^TWb$ on the vectors $v,z$, namely
\[
\begin{pmatrix} -\mu+\lambda {(u^Tv)}^2 & \lambda(u^Tv)\|z\|^2\\ \lambda(u^Tv)\|z\|^2 & \lambda\|z\|^4
\end{pmatrix},
\]
is positive semidefinite. Hence its determinant is nonnegative:
\[
0 \leq \big(-\mu+\lambda(u^Tv)^2\big)\lambda\|z\|^4 - \lambda^2(u^Tv)^2\|z\|^4
     = -\mu\lambda\|z\|^4.
\]
This means that $\lambda \leq 0$, a contradiction to the first item. Hence $u\in{\cal R}(Q)$.

3. As noted above, $Q$ has a negative eigenvalue. Suppose that $Q$ has two negative eigenvalues, counted with multiplicity, so that there exist two orthonormal
eigenvectors $v,w$ with $Qv=-\mu_1 v,~Qw=-\mu_2 w,$ where $\mu_1, \mu_2>0$ are not necessarily distinct.

Again, let $\lambda\in\Gamma_{Q,u}$ so that $W:=Q+\lambda uu^T\succeq0$. Set $\epsilon_1:=u^Tv$ and $\epsilon_2:=u^Tw$. As obtained earlier, the Gram matrix of the form $(a,b)\mapsto a^TWb$ on the vectors $v,w$ is
\[
\begin{pmatrix} -\mu_1+\lambda {\epsilon}_1^2 & \lambda {\epsilon}_1{\epsilon}_2\\ \lambda{\epsilon}_1{\epsilon}_2 & -\mu_2+\lambda {\epsilon_2}^2\end{pmatrix}.
\]
This matrix is positive semidefinite. So, $\epsilon_1, \epsilon_2 \neq0, ~\lambda \geq \mu_1/\epsilon_1^2$ and
$\lambda\geq\mu_2/\epsilon_2^2$. Since its determinant is nonnegative, we have
\[
0 \leq \mu_1\mu_2 - \lambda(\mu_1\epsilon_2^2+\mu_2\epsilon_1^2),\]
yielding
\[
\lambda \leq \frac{\mu_1\mu_2}{\mu_1\epsilon_2^2+\mu_2\epsilon_1^2}.
\]
However, note that
\[
\frac{\mu_1}{\epsilon_1^2} -\lambda \geq \frac{\mu_1}{\epsilon_1^2} - \frac{\mu_1\mu_2}{\mu_1\epsilon_2^2+\mu_2\epsilon_1^2}
= \frac{\mu_1^2\epsilon_2^2}{\epsilon_1^2(\mu_1\epsilon_2^2+\mu_2\epsilon_1^2)} > 0,
\]
a contradiction to $\lambda\geq\mu_1/\epsilon_1^2$. Hence $Q$ has at most one negative eigenvalue, counted with multiplicity, and so exactly one. In particular, this eigenvalue is simple.
\end{proof}

\begin{theorem}\label{mugamma}
Under the hypotheses of Theorem \ref{qprop1}, we have
\begin{center}
$u^TQ^{\dag}u <0$ and $\lambda_0=\min\Gamma_{Q,u}=-\dfrac{1}{u^TQ^{\dag}u}$.
\end{center}
Moreover, $x^0:=Q^{\dag}u\neq 0$ and $(Q+\lambda_0uu^T)x^0=0.$ In particular, $Q+\lambda_0uu^T$ is singular.
\end{theorem}

\begin{proof}
Using (3) of Theorem \ref{qprop1}, we may write the eigen-decomposition of $Q$ as
\begin{center}
$Q=-\mu\,vv^T + \sum_j \gamma_j {\bf w}_j{\bf w}_j^T$, with $\mu>0$ and $\gamma_j>0$.
\end{center}
Here, $v$ is such that $Qv = -\mu v,~\|v \|=1$ and $\mu > 0$, and the sum runs over the positive eigenvalues $\gamma_j$ of $Q$ (the zero eigenvalues do not appear). The vectors $v, {\bf w}_j$ for all $j$ as above, are orthonormal. Then, it follows that
\begin{center}
$Q^{\dag}=-\frac{1}{\mu}\,vv^T + \sum_j \frac{1}{\gamma_j} {\bf w}_j{\bf w}_j^T.$
\end{center}
By (2) of Theorem \ref{qprop1}, $u \in {\cal R}(Q)$. So, we may write
\begin{center}
$u=\eta v+\sum_j\eta_j{\bf w}_j,$ with $\eta=u^Tv$ and $\eta_j=u^T{\bf w}_j$.
\end{center}
Note that $\eta\neq 0$, by the proof of (1) of Theorem \ref{qprop1}.

Set $x^0:=Q^{\dag}u$ and $\mu^*:=u^TQ^{\dag}u=u^Tx^0.$ Since $u\in{\cal R}(Q),$ we have $Qx^0=QQ^{\dag}u=u.$ As $u\neq 0,$ it follows that $x^0\neq 0.$ Also,
\begin{center}
$x^0= -\frac{\eta}{\mu}v+\sum_j\frac{\eta_j}{\gamma_j}{\bf w}_j~$ and $~\mu^*= -\frac{\eta^2}{\mu}+\sum_j\frac{\eta_j^2}{\gamma_j}.    $
\end{center}
Further, for any $x$,
\begin{center}
$(x^0)^TQx = (Qx^0)^Tx=u^Tx~$ and  $~(x^0)^TQx^0=u^Tx^0=\mu^*.$    
\end{center}
We now split the proof into two cases.

\medskip
\noindent\textbf{Case 1: $u\in \rm sp\{v\}$.}

In this case, all $\eta_j=0$ and $u=\eta v$, with $\eta\neq 0$ given as above. Then
\[
x^0 = -\frac{\eta}{\mu}v \quad
\text{and} \quad \mu^* = -\frac{\eta^2}{\mu} < 0,
\]
so that
\[
-\frac{1}{\mu^*} = \frac{\mu}{\eta^2}.
\]

Let $\lambda\in\Gamma_{Q,u}$, so that $W:=Q+\lambda uu^T\succeq 0$. Since $u=\eta v$,
\[
Wv = Qv + \lambda(u^Tv)u
= -\mu v + \lambda\eta^2 v
= (-\mu+\lambda\eta^2)v.
\]
As $W\succeq0$, we must have
\[
-\mu+\lambda\eta^2\ge 0
\quad \textit{so that}\quad
\lambda\ge \frac{\mu}{\eta^2} = -\frac{1}{\mu^*}.
\]
Thus, $\Gamma_{Q,u}\subseteq \bigl[-\frac{1}{\mu^*},\infty\bigr)$ and hence
\[
\lambda_0 \ge -\frac{1}{\mu^*}.
\]
Set $\lambda^* := -\frac{1}{\mu^*} = \frac{\mu}{\eta^2}$. Since $uu^T=\eta^2vv^T,$ the eigen-decomposition given earlier yields
\begin{align*}
Q+\lambda^*uu^T & = (-\mu\,+\lambda^*{\eta}^2)vv^T + \sum_j \gamma_j {\bf w}_j{\bf w}_j^T\\
& = \sum_j \gamma_j {\bf w}_j{\bf w}_j^T.
\end{align*}
Since $\gamma_j>0,$ for all $j$ above, $Q+\lambda^*uu^T\succeq 0$, so that $\lambda^*\in\Gamma_{Q,u}$. Hence,
\[
\lambda_0\le \lambda^* = -\frac{1}{\mu^*}.
\]
We have shown that $\lambda_0 = -\frac{1}{\mu^*},$ completing Case 1. Observe that
\[
(Q+\lambda^*uu^T)v = Qv+\lambda^* \eta u = -\mu v+\lambda^*{\eta}^2v = 0,
\]
and $x^0$ is a multiple of $v.$ Thus, $(Q+\lambda_0uu^T)x^0=0$; that is, the positive semidefinite matrix $Q+\lambda uu^T$ is singular at the optimum value $\lambda = \lambda^*=\lambda_0.$

\medskip
\noindent\textbf{Case 2: $u\notin \rm sp\{v\}$.}

In this case, at least one $\eta_j\neq 0$. Hence $x^0\notin \rm sp\{v\},$ and since $x^0\neq 0,$ the vectors $v, x^0$ are linearly independent. Again, let $\lambda\in\Gamma_{Q,u}$, so that $W:=Q+\lambda uu^T\succeq 0$. We compute:
\[
v^TWv = v^TQv+\lambda {(u^Tv)}^2=-\mu +\lambda \eta^2,
\]
\[
v^TWx^0 = v^TQx^0+\lambda (v^Tu)(u^Tx^0) = \eta+\lambda \eta\mu^* = \eta(1+\lambda\mu^*),
\]
\[
(x^0)^TWx^0=(x^0)^TQx^0+\lambda {((x^0)^Tu)}^2=\mu^* + \lambda {(\mu^*)}^2=\mu^*(1+\lambda \mu^*).
\]
Thus, the Gram matrix of the form $(a,b)\mapsto a^TWb$ on the vectors $v,x^0$ is
\[
\begin{pmatrix}
-\mu + \lambda\eta^2 & \eta(1+\lambda\mu^*)\\[2pt]
\eta(1+\lambda\mu^*) & \mu^*(1+\lambda\mu^*)
\end{pmatrix}.
\]
As $W\succeq 0$, this matrix is positive semidefinite, so that its determinant is nonnegative:
\begin{align*}
0
&\le (-\mu+\lambda \eta^2)\mu^*(1+\lambda\mu^*) - \eta^2(1+\lambda\mu^*)^2 \\
&= -(1+\lambda\mu^*)(\mu\mu^*+\eta^2). \tag{$\ast$}
\end{align*}
Now
\[
\mu\mu^*+\eta^2 = \mu\sum_j\frac{\eta_j^2}{\gamma_j} =: \delta \ge 0.
\]
If $\delta=0$, then all $\eta_j=0$, so that $u=\eta v,$ a contradiction to $u\notin \rm sp \{v\}$. Thus, $\delta>0$. Then ($\ast$) forces
\[
1+\lambda\mu^*\le 0.
\]
Since $\lambda>0$, this implies that
\[
\mu^*<0
\quad\text{and}\quad
\lambda\ge -\frac{1}{\mu^*}.
\]
Thus, once again, $\Gamma_{Q,u}\subseteq [-1/\mu^*,\infty)$ and $\lambda_0\ge -1/\mu^*$.

To conclude the proof in this case, we must show that $\lambda^* := -1/\mu^*$ is feasible. Observe that, for any $\lambda\in\Gamma_{Q,u}$ and $y\perp u,$ we have
\[
0\le y^T(Q+\lambda uu^T)y = y^TQy.
\]
Let $x\in\mathbb{R}^n$ be arbitrary, and set
\[
\alpha := \frac{u^Tx}{\mu^*},\qquad y := x - \alpha x^0.
\]
Then
\[
u^Ty = u^Tx - \alpha\,u^Tx^0
= u^Tx - \alpha\mu^* = 0,
\]
so that
\begin{align*}
0 \le y^TQy
&= x^TQx - 2\alpha (x^0)^TQx + \alpha^2 (x^0)^TQx^0 \\
&= x^TQx - 2\alpha (u^Tx) + \alpha^2 \mu^* \\
&= x^TQx - \frac{(u^Tx)^2}{\mu^*}.
\end{align*}
Recalling that $\lambda^*=-1/\mu^*$, we then have, for every $x \in \mathbb{R}^n$,
\[
x^T(Q+\lambda^*uu^T)x = x^TQx - \frac{(u^Tx)^2}{\mu^*} \ge 0.
\]
Thus $Q+\lambda^*uu^T\succeq 0$, i.e., $\lambda^*\in\Gamma_{Q,u}$. We have shown that
\[
\lambda_0\le \lambda^* = -\frac{1}{\mu^*},
\]
and so $\lambda_0=-1/\mu^*.$ Finally,
\[
(Q+\lambda_0uu^T)x^0 = Qx^0+\lambda_0(u^Tx^0)u = u+\lambda_0\mu^*u = (1+\lambda_0\mu^*)\,u = 0,
\]
with $x^0\neq 0.$ Hence the matrix $Q+\lambda uu^T$ is singular at $\lambda = \lambda_0,$ in this case, too.
\end{proof}

We are now in a position to prove the second main result.

\begin{theorem}
\label{symmetric}
Let $A$ be a real symmetric matrix of rank one, and let $M_A(X)=AXA^T$. Then $M_A$ has the $Q_0$-property.
\end{theorem}

\begin{proof}
Since $A$ is symmetric of rank one, $A=\pm uu^T$ for some nonzero vector $u\in\mathbb{R}^n$, the sign being that of the unique nonzero eigenvalue of $A.$ In either case,
\[
M_A(X)=uu^TXuu^T=(u^TXu)\,uu^T,
\]
so that $M_A=M_{uu^T}.$ Hence we may, and do, assume that $A=uu^T.$

We claim that, for any $Q\in\mathbb{S}^n,$ the feasibility of $\mathrm{SDLCP}(M_A,Q)$ implies its solvability. Let $Q\in\mathbb{S}^n$ be such that $Q\nsucceq 0$ and $\mathrm{SDLCP}(M_A,Q)$ is feasible. Let $X\succeq0$ with $M_A(X)+Q\succeq0$. Setting $\gamma:=u^TXu\geq0$, this means that $\gamma\in\Gamma_{Q,u}$, proving that $\Gamma_{Q,u}\neq\emptyset.$
Let $x^0, \lambda_0$ and $\mu^*$ be as in Theorem \ref{mugamma}, so that $u^Tx^0=\mu^*\neq0$ and $\lambda_0\mu^*=-1$. Define
$$W_0:=Q+\lambda_0uu^T.$$
As $\lambda_0\in\Gamma_{Q,u}$, $W_0\succeq 0$  and
\[
W_0x^0 = Qx^0+\lambda_0 (u^Tx^0)u = u+\lambda_0\mu^*\,u = u\big(1+\lambda_0\mu^*\big)=0.
\]
Define
\[
X_0 := \lambda_0^3\,x^0(x^0)^T.
\]
Since $\lambda_0>0$, $X_0\succeq0$. Also, since  $(\mu^*)^2=\lambda_0^{-2},$
\[
u^TX_0u = \lambda_0^3(u^Tx^0)^2=\lambda_0^3(\mu^*)^2=\lambda_0,
\]
so that $M_A(X_0)=(u^TX_0u)\,uu^T=\lambda_0uu^T.$ This yields $M_A(X_0)+Q = W_0\succeq0$, showing that $X_0$ is feasible. Finally,
\[
W_0X_0 = \lambda_0^3\,(W_0x^0){x^0}^T = 0,
\]
so that $\langle X_0, M_A(X_0)+Q\rangle=\operatorname{tr}(X_0W_0)=\operatorname{tr}(W_0X_0)=0.$
Hence, $X_0$ solves $\mathrm{SDLCP}(M_A,Q)$.
\end{proof}
Next, we present an illustrative example.

\begin{example}
Let $Q = \begin{pmatrix} 1 & 2 & 1 \\ 2 & 1 & 1 \\ 1 & 1 & \frac{2}{3} \end{pmatrix}, 
u = \begin{pmatrix} 1 \\ 5 \\ 2 \end{pmatrix}$ and $x = \begin{pmatrix} 4 \\ 0 \\ -3 \end{pmatrix},$ so that $Qx=u.$ To compute $u^TQ^{\dag}u$ one does not need the Moore-Penrose inverse $Q^{\dag}.$ It suffices to use a particular solution of the equation $Qx=u,$ since $u \in {\cal R}(Q)$ and $Q$ is symmetric. In this case, we then have $u^TQ^{\dag}u=u^Tx,$ where $x$ is given as above. Thus, $u^TQ^{\dag}u=-2$ so that $\lambda_0=1/2.$ Observe that 
$$Q + \lambda u u^T = \begin{pmatrix} 1 + \lambda & 2 + 5\lambda & 1 + 2\lambda \\ 2 + 5\lambda & 1 + 25\lambda & 1 + 10\lambda \\ 1 + 2\lambda & 1 + 10\lambda & \frac{2}{3} + 4\lambda \end{pmatrix}.$$ 
The leading principal submatrix of order $2 \times 2$ has the determinant value $6\lambda -3.$ In order for the given matrix to be positive semidefinite, we must have $\lambda \geq \frac{1}{2},$ confirming the choice above. One may verify that 
\[
Q + \frac{1}{2} u u^T 
=
\begin{pmatrix}
\frac{3}{2} & \frac{9}{2} & 2 \\[4pt]
\frac{9}{2} & \frac{27}{2} & 6 \\[4pt]
2 & 6 & \frac{8}{3}
\end{pmatrix},
\]
is positive semidefinite.     
\end{example}

In the next example, we show that the conclusion of Theorem \ref{symmetric} does not extend to the case, when $A$ is (symmetric,) singular and $\rm rank(A) \geq 2.$

\begin{example}\label{counter1}
For $
A = \begin{pmatrix} 1 & ~~0 & ~~0 \\ 0 & -1 & ~~0 \\ 0 & ~~0 & ~~0 \end{pmatrix}, Q = \begin{pmatrix} -1 & ~~2 & ~~0 \\ ~~2 & -1 & ~~0 \\ ~~0 & ~~0 & ~~1 \end{pmatrix}$ and for $X=(x_{ij}) \in \mathbb{S}^3$, we have 
\[
M_A(X) = \begin{pmatrix} ~~x_{11} & -x_{12} & ~~0 \\ -x_{12} & ~~x_{22} & ~~0 \\ ~~0 & ~~0 & ~~0 \end{pmatrix}
\]
and so \[
W:=M_A(X) + Q = \begin{pmatrix} ~~x_{11}-1 & -x_{12}+2 & ~~0 \\ -x_{12}+2 & ~~x_{22}-1 & ~~0 \\ 0 & ~~0 & ~~1 \end{pmatrix}.
\]
Let  
\[
X_0 = \begin{pmatrix} 5 & 2 & 0 \\ 2 & 5 & 0 \\ 0 & 0 & 0 \end{pmatrix},
\]
so that 
\[
W_0: = M_A(X_0)+Q = \begin{pmatrix} 4 & 0 & 0 \\ 0 & 4 & 0 \\ 0 & 0 & 1 \end{pmatrix} \succeq 0,
\]
showing that $\mathrm{SDLCP}(M_A,Q)$ is feasible, for the chosen $Q$. Next, let $X$ solve $\mathrm{SDLCP}(M_A,Q)$. Then, $XW=0$ and, since the third column of $W$ equals $(0,0,1)^T$, we have $x_{13}=x_{23}=x_{33}=0$. We may then rewrite 
\[
X = \begin{pmatrix} X_1 & 0 \\ 0 & 0 \end{pmatrix}, \quad \textit{with} \quad X_1 = \begin{pmatrix} x & y \\ y & z \end{pmatrix} \succeq 0.
\]
The condition $XW=0$ is equivalent to $X_1 W_1 = 0$, where
\[
W_1 = \begin{pmatrix} x-1 & 2-y \\ 2-y & z-1 \end{pmatrix} \succeq 0.
\]
Note that we must have $x, z \geq 1,$ so that $x+z \geq 2.$ Then, $X_1 W_1 = 0$ expands into  
$$x(x-1) + y(2-y) = 0; x(2-y) + y(z-1) = 0;
z(z-1) + y(2-y) = 0. $$ 
The first and the third yield 
$
z^2-z-x^2+x = 0$; equivalently, $(z-x)(z+x-1)=0.$ The second factor is positive and so $x=z.$ Set $t:=x=z.$  
The second equation simplifies to $t(2-y)+y(t-1) = 2t-y=0$, so that $y=2t$. Substituting into the first equation:
\[
0=t^2-t+2t(2-2t) = t^2 - t + 4t - 4t^2 = -3t^2+3t = -3t(t-1),
\]
so that $t=0$ or $t=1$. When $t=0$, we have $W_1=\begin{pmatrix}-1&2\\2&-1\end{pmatrix}$, and when $t=1, y=2$, $X_1=\begin{pmatrix}1&2\\2&1\end{pmatrix}$. Neither of these matrices is positive semidefinite. 
\end{example}

\section{Acknowledgements}
The author thanks Prof. M.S. Gowda (UMBC, USA) and Dr. I. Jeyaraman (NIT, Trichy) for their comments and suggestions. These resulted in a much clearer presentation of the results.

\end{document}